\documentclass[a4paper,11pt,twoside,reqno,nonamelimits]{amsart}
\usepackage[colorlinks=true,urlcolor=blue,linkcolor=blue,citecolor=blue]{hyperref}
\usepackage{mathptmx}
\usepackage{amsmath}
\usepackage{amssymb}
\usepackage{amsfonts}
\usepackage{amscd}
\usepackage{amsthm}
\usepackage{amsxtra}
\usepackage[all,pdf]{xy}  \CompileMatrices     
\usepackage{mathrsfs}
\usepackage{nicefrac}
\usepackage{bbold}
\usepackage{graphicx}
\usepackage{xcolor}
\usepackage{enumerate}
\usepackage{wrapfig}

\usepackage[norelsize,boxed,noend,linesnumbered]{algorithm2e}\DontPrintSemicolon

\theoremstyle{plain}
\newtheorem{theorem}{Theorem}
\newtheorem*{theorem*}{Theorem}

\newtheorem*{proposition*}{Proposition}

\newtheorem*{corollary*}{Corollary}

\newtheorem{lemma}[theorem]{Lemma}
\newtheorem*{lemma*}{Lemma}

\newtheorem*{observation*}{Observation}

\newtheorem*{conjecture*}{Conjecture}

\newtheorem*{question*}{Question}

\newtheorem*{questions*}{Questions}

\newtheorem*{problem*}{Problem}

\newtheorem*{problems*}{Problems}

\newtheorem*{openproblem*}{Open Problem}

\theoremstyle{definition}

\newtheorem*{definition*}{Definition}

\newtheorem*{example*}{Example}

\newtheorem*{exercise*}{Exercise}

\theoremstyle{remark}

\newtheorem*{remark*}{Remark}

\newtheorem*{remarks*}{Remarks}

\newtheorem*{claim*}{Claim}

\newcommand{\subclass}[1]{}

\newcommand{\enumTi}[1]{\renewcommand{\theenumi}{#1}}

\newcommand{\alphenumi}{\enumTi{\alph{enumi}}}

\newcommand{\romenumi}{\enumTi{\roman{enumi}}}

\newlength{\hspaceforlengthglumpf}

\newcommand{\comment}[1]{\text{\footnotesize[#1]}}

\newcommand{\lt}{\left}
\newcommand{\rt}{\right}

\newcommand{\floor}[1]{\lt\lfloor{#1}\rt\rfloor}

\newcommand{\nfrac}[2]{{\nicefrac{#1}{#2}}}

\newcommand{\NN}{\mathbb{N}}

\newcommand{\ZZ}{\mathbb{Z}}

\newlength{\algotabbingwidth}
\newcounter{mysaveenumi}

\newcommand{\myparagraphwskip}[1]{\smallskip\paragraph{#1}}

\newcommand{\mypar}{\par\medskip\noindent}

\DeclareMathOperator{\dD}{d\mspace{-2mu}d}
\DeclareMathOperator{\dyck}{p\mspace{-2mu}d}
\DeclareMathOperator{\RHS}{RHS}

\begin{document}
\title[1-Ascents in dispersed Dyck paths]{Short Note on the Number of 1-Ascents in Dispersed Dyck Paths}%
\author{Kairi Kangro}%
\author{Mozhgan Pourmoradnasseri}%
\thanks{MP is recipient of the Estonian IT Academy Scholarship, and supported by the Estonian Research Council, ETAG (\textit{Eesti Teadusagentuur}).}
\author{Dirk Oliver Theis}%
\thanks{DOT is supported by the Estonian Research Council through PUT Exploratory Grant \#620; and by the European Regional Development Fund through the Estonian Center of Excellence in Computer Science, EXCS\mbox{}.}
\address{Dirk Oliver Theis\\
  University of Tartu\\
  Insitute of Computer Science\\
  J.~Liivi~2\\
  50409 Tartu\\
  Estonia.}%
\urladdr{http://ac.cs.ut.ee/people/dot/}%
\email{dotheis@ut.ee}%

\begin{abstract}
  A dispersed Dyck path (DDP) of length~$n$ is a lattice path on $\NN\times\NN$ from $(0,0)$ to $(n,0)$ in which the following steps are allowed: ``up'' $(x,y)\to(x+1,y+1)$; ``down'' $(x,y)\to(x+1,y-1)$; and  ``right'' $(x,0)\to(x+1,0)$.  An ascent in a DDP is an inclusion-wise maximal sequence of consecutive up steps.  A 1-ascent is an ascent consisting of exactly 1 up step.\\
  We give a closed formula for the total number of 1-ascents in all dispersed Dyck paths of length~$n$, \#A191386 in Sloane's OEIS.  Previously, only implicit generating function relations and asymptotics were known.
\end{abstract}

\date{Mon Dec 14 03:58:18 EST 2015}

\maketitle

\section{Introduction}\label{sec:intro}
A \textit{Dyck path of length~$n$} is a lattice path on $\NN\times\NN$ from $(0,0)$ to $(n,0)$ in which the following steps are allowed:
\begin{itemize}
\item[$(+1,+1)$:] up step  $(x,y)\to(x+1,y+1)$;
\item[$(+1,-1)$:] down step $(x,y)\to(x+1,y-1)$, if $y>0$.
\end{itemize}
A \textit{dispersed Dyck path (DDP) of length~$n$} is a lattice path on $\NN\times\NN$ from $(0,0)$ to $(n,0)$ in which the following additional steps are allowed:
\begin{itemize}
\item[$(+1,0)$:] right step $(x,y)\to(x+1,y)$, \textbf{only if $y=0$.}
\end{itemize}
Put differently, dispersed Dyck paths are Motzkin paths with no $(+1,0)$-steps at positive heights.

An \textit{ascent} in a DDP is an inclusion-wise maximal sequence of consecutive up steps;  a \textit{1-ascent} is an ascent consisting of exactly 1 up step.

We give a closed formula for the total number~$A(n)$ of 1-ascents in all dispersed Dyck paths of length~$n$.  This is \href{https://oeis.org/A191386}{\#A191386} in Sloane's Online Encyclopedia of Integer Sequences, OEIS.  Previously, only implicit generating function relations and the following asymptotics were known~\cite{Kotesovec14}:
\begin{equation*}
  A(n) \sim \sqrt{\nfrac{n}{\pi}} \, \lt(1 + \sqrt{\nfrac{\pi}{2n}}\rt) \; 2^{n-5/2}.
\end{equation*}
We prove the following.
\begin{theorem}\label{thm:1asc}
  For $n \ge 1$, the number of 1-ascents in all DDPs of length~$n+2$ is
  \begin{equation*}
    A(n) \ = \;\; 2^{n-1}   +   \frac{n + 1}{2}\; \binom{ n }{ \lfloor \nfrac{n}{2} \rfloor }.
  \end{equation*}
\end{theorem}

\section{Proof of the theorem}\label{sec:proof}
Denote by $\dD(n)$ the number of DDPs of length~$n$.  We repeat the following easy, folklore, fact for the sake of completion.
\begin{lemma}[Folklore]\label{lem:total_dD}
  For $n \ge 0$,
  \begin{equation*}
    \dD(n) = \binom{ n }{ \lfloor \nfrac{n}{2} \rfloor }
  \end{equation*}
\end{lemma}
\begin{proof}
  The proof is by a variant of the reflection method: We give a bijection between the set of all DDP of length~$n$ and the set of all paths from $(0,0)$ to $(n, -n \% 2)$ in $\ZZ\times \ZZ$ using only up steps $(+1,+1)$ and down steps $(+1,-1)$.  Here $n\%2$ denotes the remainder upon division by~2.
  We call the latter paths ``plain paths'', PP.

  Indeed, a DDP can be obtained from a PP in the following way.  Wherever~$p$ takes a step $(x,0)\to(x+1,-1)$, let $x'$ be the next occurrence (i.e., smallest $x' > x$) of a step $(x',-1)\to(x'+1,0)$, if such a step exists, or $x':=n$ otherwise.  We the reflect the PP in the $x$-axis interval $[x+1,x']$ at the line $y=\nfrac12$, and replace the down step $(x,0)\to(x+1,-1)$ by a right step $(x,0)\to(x+1,0)$.  If $x'<n$, we also replace the up step $(x',-1)\to(x'+1,0)$ by a right step $(x',0)\to(x'+1,0)$.

  This is performed for all parts of the PP which venture below the $x$-axis.  We leave to the reader the details of verifying that this indeed defines a bijection between the PPs and the DDPs.
\end{proof}

We will use the following sequences.  Denote\footnote{We reserve upper-case letters for ``total number of \dots in \dots'' counters.} %
by
\begin{equation*}
  \lt.
  \begin{array}[c]{c}
    U(n) \\
    D(n) \\
    R(n) \\
    A(n) \\
  \end{array}
  \rt\}
  \text{ the total number of }
  \lt\{
  \begin{array}[c]{l}
    \text{up steps}\\
    \text{down steps}\\
    \text{right steps}\\
    \text{1-ascents}\\
  \end{array}
  \rt.
\end{equation*}
in all DDPs of length~$n$.  Note that $U(n)=D(n)$.
Moreover, let $\dyck(n)$ denote the number of (proper) Dyck paths of length~$n$ (which is~0, if $n$ is odd).

The next lemma states that the total number of right steps in all DDPs of length $2k$ is twice the total number of right steps in all DDPs of length $2k-1$.
\begin{lemma}\label{lem:right_right-recursion}
  For $k\ge 1$, $\displaystyle R(2k) = 2R(2k-1)$.
\end{lemma}
\begin{proof}
  For $k=1$, the equation holds.  Let $k\ge 2$.

  There are an even number of right steps in a DDP of length~$2k$.  Numbering them from left to right (starting from~0), we count the right steps in pairs
  \begin{equation*}
    \text{ ``($2j$th right step, $(2j+1)$th right step)'' }
  \end{equation*}
  and then multiply by two for the total number.  For every pair of right steps, there are an even number (possibly 0) of steps between
  \begin{enumerate}[(a)]
  \item $0$ and the the start of the first right step in the pair
  \item the end point of the first right step in the pair and the beginning of the second one
  \item end point of the second right step of the pair, and the point $(2k,0)$.
  \end{enumerate}
  Hence, every pair $(x,x')$ with
  \begin{enumerate}[(a)]
  \item $x \ge 0$ and even
  \item $x'-x \ge 1$ and odd
  \item $x' \le 2k$ and $2k-x'-1$ even
  \end{enumerate}
  is a potential pair of right steps $( (x,0)\to(x+1,0), (x',0)\to(x'+1,0))$ in a DDP of length~$2k$.

  We identify the DDPs of length $2k-1$ with the subset of those DDPs of length~$2k$ for which the last step is a right step: $(2k-1,0)\to(2k,0)$.  In this manner, the right steps in the DDPs of length $2k-1$ also come in pairs, but the second right step in the last pair is always $(2k-1,0)\to(2k,0)$.

  For $(x,x')$ satisfying the conditions above, denote by $R_{x,x'}$ the total number of pairs of right steps $( (x,0)\to(x+1,0), (x',0)\to(x'+1,0))$ in all DDPs of length~$2k$, and by $R^-_{x,x'}$ the total number of pairs of right steps $( (x,0)\to(x+1,0), (x',0)\to(x'+1,0))$ in all DDPs of length~$2k$ ending in a right step $(2k-1,0)\to(2k,0)$.

  We have
  \begin{equation*}
    R(2k) = \sum_{x<x'\le n} 2 R_{x,x'}, \quad\text{ and }\quad  R(2k-1) = \sum_{x<x'<n} 2 R^-_{x,x'} + \sum_{x<x'=n} R^-_{x,x'},
  \end{equation*}
  where the sums extend over all $x,x'$ satisfying the conditions (a--c) above.

  We will show that for $x<x'<n$, $R_{x,x'} = 2R^-_{x,x'}$.  Since, $R_{x,2k-1} = R^-_{x,2k-1}$ by definition, this concludes the proof of the lemma.

  The equation $R_{x,x'} = 2R^-_{x,x'}$, for $x' < n$ is proved as follows.  Using the fact that
  \begin{equation}\label{eq:binom-half-even}
    \binom{\ell}{\ell/2} = 2\binom{\ell-1}{\ell/2-1}
  \end{equation}
  we have
  \begin{align*}
    R_{x,x'}
    &= \dD(x)\cdot \dyck(x'-x-1) \cdot \dD(2k-x'-1) \\
    &= \dD(x)\cdot \dyck(x'-x-1) \cdot \binom{2k-x'-1}{\lfloor\frac{2k-x'-1}{2}\rfloor} &&\comment{by Lemma~\ref{lem:total_dD}} \\
    &= \dD(x)\cdot \dyck(x'-x-1) \cdot 2 \binom{2k-x'-2}{\lfloor\frac{2k-x'-2}{2}\rfloor} &&\comment{by \eqref{eq:binom-half-even} since $2k-x'-1$ is even, (c)} \\
    &= 2 R^-_{x,x'}.
  \end{align*}
  This completes the proof of the lemma.
\end{proof}

The next lemma states that the total number of up steps (or down steps, since the number of ups and downs has to be equal) in all DDPs of length $2k+1$ is twice the total number of up steps in all DDPs of length $2k$.
\begin{lemma}\label{lem:up_up-recursion}
  For $k\ge 1$, $\displaystyle U(2k+1) = 2 U(2k)$.
\end{lemma}
\begin{proof}
  Again we identify DDPs of length $2k$ with the subset of those DDPs of length $2k+1$ whose last step is a right step $(2k,0)\to(2k+1,0)$.  Denote by $U^-$ the total number of up steps in all DDPs of length $2k+1$ whose last step is a right step $(2k,0)\to(2k+1,0)$, and note that $U^- = U(2k)$.  What remains to be shown is that the total number of up steps in all those DDPs of length $2k+1$ whose last step is a down step $(2k,1)\to(2k+1,0)$ is equal to $U(2k)$.

  There are two types of up steps in the DDPs of length $2k$:
  \begin{enumerate}[(a)]
  \item\label{enum:upup:a} those which come from proper Dyck paths (DDPs which only contain up and down steps); and
  \item\label{enum:upup:b} those which come from DDPs which contain a right step.
  \end{enumerate}
  Similarly, there are two types of down steps in those DDPs of length $2k+1$, whose last step is a down step: \setcounter{mysaveenumi}{\theenumi}
  \begin{enumerate}[(a)]\setcounter{enumi}{\themysaveenumi}
  \item\label{enum:upup:aplus} the last down step $(2k,1)\to(2k+1,0)$; and
  \item\label{enum:upup:bplus} all the others.
  \end{enumerate}
  As for (\ref{enum:upup:a}), each of these paths contains exactly~$k$ up steps, so their total number is $k\, c(2k)$, where $c(2k) := \frac{1}{k}\binom{2k}{k-1}$ is the $k$th Catalan number.  Hence, their number is $\binom{2k}{k-1}$.
  As for (\ref{enum:upup:aplus}), their number is coincides with the number of DDPs of length $2k+1$ ending in a down step, which equals the total number of DDPs of length $2k+1$ \textit{minus} the number of DDPs of length $2k+1$ ending in a right step.  The last number is just the number of DDPs of length~$2k$, so, by invoking Lemma~\ref{lem:total_dD}, we obtain, for the down steps in~(\ref{enum:upup:aplus}),
  \begin{equation*}
    \binom{2k+1}{k} - \binom{2k}{k} = \binom{2k}{k-1}.
  \end{equation*}
  Thus, the numbers in (\ref{enum:upup:a}) and (\ref{enum:upup:aplus}) are equal.

  \mypar%
  As for (\ref{enum:upup:b}) and (\ref{enum:upup:bplus}), we define a bijection between the type-(\ref{enum:upup:bplus}) down steps and the type-(\ref{enum:upup:b}) up steps.

  For a DDP of length $2k+1$ whose last step is a down step $(2k,1)\to(2k+1,0)$, find the greatest~$x$ for which $(x,0)\to(x+1,1)$ is a step in the DDP.  Then replace the up step $(x,0)\to(x+1,1)$ by a right step $(x,0)\to(x+1,0)$, and delete the final (down) step.  The resuls is a DDP of length $2k$ which as at least~1 right step.  For the inverse, for a DDP of length $2k$ which has at least~1 right step, take the greatest~$x$ for which $(x,0)\to(x+1,0)$ is a right step.  Replace this right step by an up step $(x,0)\to(x+1,1)$, and add a tailing down step $(2k,1)\to(2k+1,0)$, to obtain a DDP of length $2k+1$.

  It is easy to verify that this operation defines a bijection taking the down steps in the longer DDPs onto the down steps in the $2k$-DDPs, and hence between the sets defining (\ref{enum:upup:b}) and (\ref{enum:upup:bplus}).
\end{proof}

Combining the previous two lemmas, we obtain the following expression for the total number of right steps:
\begin{lemma}\label{lem:right-closed}
  For $n\ge 1$, we have
  \begin{equation}\label{eq:right-closed}
    R(n) = 2^n -\binom{n}{\lfloor \nfrac{n}{2} \rfloor}
  \end{equation}
\end{lemma}
\begin{proof}
  The equation is readily verified for $n=1,2$.  We show that both $R(\cdot)$ and $\RHS(\cdot)$ satisfy the same recursive relation.

  From Lemma~\ref{lem:right_right-recursion}, we know that $R(2k+1) = 2 R(2k)$.

  From previous lemma and the fact that
  \begin{equation}\label{eq:total-total-steps}\tag{\textasteriskcentered}
    n\dD(n) = R(n) + U(n) + D(n) = R(n) +2U(n),
  \end{equation}
  we obtain
  \begin{align*}
    R(2k+1)
    &= (2k+1)\dD(2k+1) - 2 U(2k+1) &&\comment{by~\eqref{eq:total-total-steps}}\\
    &= (2k+1)\dD(2k+1) - 4 U(2k)   &&\comment{by Lemma~\ref{lem:up_up-recursion}}\\
    &= (2k+1)\dD(2k+1) - 2 \bigl[ 2k\dD(2k) - R(2k) \bigr] &&\comment{by~\eqref{eq:total-total-steps} again}\\
    &= 2R(2k) + (2k+1)\binom{2k+1}{k} - 4k \binom{2k}{k}. &&\comment{by Lemma~\ref{lem:total_dD}}
  \end{align*}

  Using $(k+1)\binom{2k+1}{k} = (2k+1)\binom{2k}{k}$, we find that the RHS of~\eqref{eq:right-closed} satisfies the equation $\RHS(2k+1) = 2\RHS(2k) + (2k+1)\binom{2k+1}{k} - 4k \binom{2k}{k}$, too.
\end{proof}

As the last ingredient, we prove that the total number of 1-ascents in all DDPs of length $n+2$ equals the number of DDPs of length~$n$, \textit{plus} the total number of down steps in all DDPs of length~$n$, \textit{plus} the total number of right steps in all DDPs of length~$n$.

\begin{lemma}\label{lem:1ascent-formula}
  For $n\ge 1$
  \begin{equation*}
    A(n+2) = \dD(n) + D(n) + R(n).
  \end{equation*}
\end{lemma}
\begin{proof}
  First of all, a 1-ascent can never be the last step in a DDP, since it has to end on the $x$-axis and never goes below the $x$-axis.  For the same reason, a 1-ascent can never be followed by a right step, since we can only do right steps on the $x$-axis.  Hence, every 1-ascent has to be followed by a descent.  Also, by the definition of a 1-ascent, it has to be preceded by either a descent, a right step, or the beginning of the path.  It is also clear that removing the 1-ascent with its following descent gives us a valid DDP with length shorter by two.

  Hence, we can pair each 1-ascent in a DDP with length $n+2$ with either a beginning, a descent, or a right step of a DDP with length~$n$---the one that preceded the 1-ascent in the DDP with length $n + 2$ before its removal.

  Since adding an ascent-descent pair either in the beginning, after a descent of after a right step in a DDP of length~$n$ always gives us a valid 1-ascent in a DDP of length $n + 2$, we get that $A(n + 2) = R(n) + D(n) + \dD(n)$, as claimed.
\end{proof}

We are now ready to prove Theorem~\ref{thm:1asc}.
\begin{proof}[Proof of Theorem~\ref{thm:1asc}.]
  Here, we use the facts that $\dD(n) = \binom{n}{\lfloor n/2 \rfloor}$, $R(n)= 2^n-\binom{n}{\floor n/2\rfloor}$.

  We will use equation~\eqref{eq:total-total-steps} from the proof of Lemma~\ref{lem:up_up-recursion} again.  Using Lemma~\ref{lem:1ascent-formula}, we have
  \begin{align*}
    A(n+2)
    &= R(n) + D(n) + \dD(n)                                                                                  &&\comment{by Lemma~\ref{lem:1ascent-formula}}\\
    &= R(n) + (n\dD(n)-R(n))/2 + dD(n)                                                                       &&\comment{by~\eqref{eq:total-total-steps}}\\
    &= \frac12 \bigl( R(n) + (n+2) \dD(n) \bigr)                                                             \\
    &= \frac12 \lt( 2^n - \binom{n}{\lfloor n/2 \rfloor} + (n+1)\binom{n}{\lfloor n/2 \rfloor} \rt)          &&\comment{by Lemmas \ref{lem:total_dD} and~\ref{lem:right-closed}}\\
    &= 2^{n-1} + \frac{n+1}{2} \binom{n}{\lfloor n/2 \rfloor}.
  \end{align*}
  This completes the proof of the theorem.
\end{proof}

\section{Conclusion}
An alternative way to prove Lemma~\ref{lem:right-closed} would be to use the equation
\begin{equation}\label{eq:A045621}
  2^n - \binom{n}{\lfloor \nfrac{n}{2} \rfloor}
  =
  \sum_{k=0}^{n-1} \binom{k}{\lfloor \frac{k}{2} \rfloor}  \binom{n-k-1}{\lfloor \frac{n-k-1}{2} \rfloor};
\end{equation}
in OEIS this is sequence \#A045621~\cite{A045621}.  Indeed, it can be shown that the total number of right steps in all DDPs of length~$n$ can be counted as
\begin{equation*}
  R(n) = \sum_{k=0}^{n-1} \binom{k}{\lfloor \frac{k}{2} \rfloor}  \binom{n-k-1}{\lfloor \frac{n-k-1}{2} \rfloor},
\end{equation*}
which, together with~\eqref{eq:A045621} implies equation~\eqref{eq:right-closed} in Lemma~\ref{lem:right-closed}.  Moreover, Lemma~\ref{lem:right-closed} implies \eqref{eq:A045621}.

We chose to present the proof which is based entirely on lattice path arguments.

There are several other statistics about dispersed Dyck paths which are of interest in the field.  It should now be checked whether closed-form expressions can be achieved for other ascent-related statistics, such as the total number of $k$-ascents, for $k>1$, or the number of dispersed Dyck paths of length $n$ with~$t$ ascents of length~1.

\smallskip%
\myparagraphwskip{\bf Acknowledgements}%
The 3rd author would like to thank Sebastian Schindler, a Masters student at the University of Magdeburg, for computationally verifying the equation in Theorem~\ref{thm:1asc} for $n \le 372$ as part of a software project.

The authors would like to thank Emeric Deutsch for his comments on an earlier version of this paper, in particular for pointing out the connection to A045621.

\providecommand{\bysame}{\leavevmode\hbox to3em{\hrulefill}\thinspace}
\providecommand{\MR}{\relax\ifhmode\unskip\space\fi MR }
\providecommand{\MRhref}[2]{%
  \href{http://www.ams.org/mathscinet-getitem?mr=#1}{#2}
}
\providecommand{\href}[2]{#2}

\end{document}